\documentclass[12pt, a4paper]{amsart}
\usepackage{mathrsfs,amssymb,amscd}

\newtheorem{theorem}{Theorem }[section]
\newtheorem{conjecture}{Conjecture}[section]
\newtheorem{assumption}{Assumption}[section]
\newtheorem{lemma}[theorem]{Lemma}
\newtheorem{corollary}[theorem]{Corollary}
\newtheorem{proposition}[theorem]{Proposition}
\theoremstyle{definition}

\theoremstyle{remark}
\newtheorem{remark}[theorem]{Remark}

\numberwithin{equation}{section}

\renewcommand{\Re}{\operatorname{Re}} 
\renewcommand{\Im}{\operatorname{Im}} 

\DeclareMathOperator*{\slim}{s-lim}
\DeclareMathOperator*{\res}{res}

\def \laplace {{\mathscr{L}}}
\def \new {{{\textup{new}}}}
\def \old {{{\textup{old}}}}
\def \G {{\Gamma}}
\def \g {{\gamma}}
\def \F {{\mathcal{F}}}
\def \l {{\lambda}}
\def \a {{\mathfrak{a}}}

\def \f {{\varphi}}

\def \R {{\mathbb R}}
\def \H {{\mathbb H}}
\def \N {{\mathbb N}}
\def \C {{\mathbb C}}
\def \Z {{\mathbb Z}}

\def \e {{\epsilon }}
\def \GinfmodG {{\Gamma_{\!\!\infty}\!\!\setminus\!\Gamma}}
\def \GmodH {{\Gamma\setminus\H}}

\def \sl  {\hbox{SL}_2(\mathbb Z)}
\def \slr  {\hbox{SL}_2(\mathbb R)}
\def \psl  {\hbox{PSL}_2(\mathbb R)}
\def \GmodH {\G\backslash\mathbb{H}}

\newcommand{\mattwo}[4]
{\left(\begin{array}{cc}
                        #1  & #2   \\
                        #3 &  #4
                          \end{array}\right) }

\newcommand{\rum}[1] {{L}^2\left( #1\right)}
\newcommand{\norm}[1]{\left\lVert #1 \right\rVert}
\newcommand{\abs}[1]{\left\lvert #1 \right\rvert}
\newcommand{\inprod}[2]{\left \langle #1,#2 \right\rangle}
\newcommand{\tr}[1] {\hbox{tr}\left( #1\right)}

\begin{document}
\date{\today}
\thanks{The author is funded by a Steno Research Grant  from The Danish Natural Science Research Council}
\title[Small eigenvalues and residual eigenvalues]{On Selberg's small eigenvalue conjecture and residual eigenvalues}
\author{Morten S. Risager}
\address{Department of Mathematics, University of Aarhus, Ny Munkegade Building 530, 8000 Aarhus C, Denmark}
\email{risager@imf.au.dk}
\begin{abstract}We show that Selberg's eigenvalue conjecture concerning small eigenvalues of the automorphic Laplacian for congruence groups is equivalent to a conjecture about the non-existence of residual eigenvalues for a perturbed system. We prove this using a combination of methods from asymptotic perturbation theory and number theory.\end{abstract}
\maketitle

\section{Introduction}

Let $\G\subseteq \sl$ be a congruence group and let $0=\lambda_0< \l_1\leq \l_2
\leq \l_3\ldots \l_i\to \infty$ be the eigenvalues of the automorphic Laplacian on $\rum \GmodH$ induced from the Laplace operator
\begin{equation*}
  \Delta=-y^2\left(\frac{\partial^2}{\partial x^2}+\frac{\partial^2}{\partial y^2}\right).
\end{equation*}
 An eigenvalue $0<\lambda<1/4$ is called a \emph{small} eigenvalue. 
In a celebrated paper  \cite{Selberg:1965aa} Selberg conjectured the following:
\begin{conjecture}\label{selbergconjecture} The Laplacian for a congruence group has no small eigenvalues, i.e. $\lambda_1\geq 1/4$.
\end{conjecture}
In this paper we prove that Selberg's conjecture is true if and only if a \lq twisted\rq{} Laplacian behaves sufficiently well: Consider the Hecke congruence groups $\G=\G_0(q)$, and the automorphic Laplacian with characters induced from a Dirichlet character $\chi_0$ modulo $q$.

\begin{equation*}
  \chi_0'\left(\mattwo{a}{b}{c}{d}\right)=\chi_0(d), \qquad \textrm{for }\mattwo{a}{b}{c}{d}\in \G_0(q)
\end{equation*}
Consider now the set $M_2^\infty(\G)$ of holomorphic forms of weight 2 which are cuspidal at infinity i.e. 
\begin{equation*}
  M_2^\infty(\G)=\{f\in M_2(\G) \vert \int_0^1f(z)dx=0\}
\end{equation*} where $M_2(\G)$ is the standard set of weight 2 holomorphic forms. We emphasize that these are not necessarily cusp forms since we only assume cuspidality \emph{at infinity}. Fix $f\in  M_2^\infty(\G)$. We then construct a family of characters (parametrized by $\e\in \R$) in the following way: 
\begin{equation}
  \label{eq:2}
  \begin{array}{llccc}\chi_\e&:&\G&\to& S^1\\
 &&\g&\mapsto&\chi_0'(\g) \exp{\left(2\pi i\e\Re \int_{z_0}^{\g z_0}f(z)dz\right)}
 \end{array}.
\end{equation}
This definition does not depend on the choice of  $z_0\in \H$. We may then consider the family of Laplacians $A(\G,\e)$ defined as the closure of the essentially selfadjoint operator defined by $\Delta g$ on smooth functions $g:\H\to\C$ such that $g(\g z)=\chi_\e(\gamma)g(z)$ for all $\g\in\G$ and such that  $g,\Delta(g)$ are square integrable. 

We propose the following conjecture:
\begin{conjecture}\label{myconjecture} For every $\G=\G_0(q)$, $\chi_0$, and $f$ as above the following holds: For every $0<a\leq b<1/4$ there exist $\e_0>0$ such that when $\abs{\e}\leq\e_0$ the operator $A(\G,\e)$ has no \emph{residual} eigenvalues in $[a,b]$.
\end{conjecture}
 We refer to section \ref{basics} for the notion of residual spectrum. It is well known (dating back to Selberg) that there are no residual eigenvalues in $]0,1/4[$ when $\e=0$.  It is not too surprising that Conjecture \ref{myconjecture} follows from Selberg's Conjecture \ref{selbergconjecture}. In fact a similar phenomenon happens whenever we have a group without small eigenvalues. What is much more surprising is that the opposite implication is true as well. I.e. we have the following: 
\begin{theorem} \label{highlow} Conjecture \ref{selbergconjecture} is true if and only if Conjecture \ref{myconjecture} is true.
\end{theorem}
Hence Selberg's conjecture about \emph{cuspidal} eigenvalues may be reformulated entirely in terms of \emph{residual} eigenvalues (of a perturbed system). We believe that the fact that  Conjecture \ref{myconjecture} implies  Conjecture \ref{selbergconjecture} is an \lq arithmetic statement\rq{} in the sense that it is tied up with the arithmeticity of the unperturbed system. Certainly the proof is arithmetic in nature: It uses Hecke-operators, non-vanishing of character twists of $L$-functions etc.

\begin{remark}We make a few comments concerning Selberg's Conjecture \ref{selbergconjecture}.  It is essential that we are considering congruence groups: For general Riemann surfaces of finite volume it is know that  small eigenvalues may occur -- both residual and cuspidal eigenvalues. See e.g. \cite{Selberg:1965aa, Randol:1974aa, Zograf:1983aa}. For general congruence groups Conjecture \ref{selbergconjecture}  is best possible since the eigenvalue 1/4 may be constructed from certain even 2-dimensional Galois representations.  

 We provide a short list of results from the history of the small eigenvalue conjecture: Selberg \cite{Selberg:1965aa} proved $\l_1\geq 3/16\approx 0.188$ using Weils bound on exponential sums, which is a consequence of the Riemann hypothesis for the zeta-function of a curve over a finite field. Jacquet \& Gelbart  \cite{Gelbart:1978aa} proved $\l_1>3/16$ and  Iwaniec \cite{Iwaniec:1990aa}  proved  a density estimate on the number of small eigenvalues. For groups of small level Huxley \cite{Huxley:1985aa} confirmed the conjecture for level $q\leq 18$. More recently Booker \& Str{\"o}mbergsson \cite{Booker:2007aa} showed that there are no small eigenvalues if $q\leq 857$ and squarefree.

A major breakthrough was made by Luo, Rudnick \& Sarnak
\cite{Luo:1995aa, Luo:1999aa} who proved that $\l_1>171/748\approx
0.218$ (and more generally found  bounds towards the Ramanujan
conjecture for $GL_n$). Iwaniec \cite{Iwaniec:1996aa} then used some
of the same ideas to show $\l_1>10/49\approx 0.204$  using a proof
which -- in contrast to the proof due to Luo, Rudnick and Sarnak -- is entirely in a $GL_2$ framework.  Later Kim \& Shahidi proved $\lambda_1\geq 66/289\approx 0.228$ \cite{Kim:2002aa}, and Kim \& Sarnak proved $\lambda_1\geq 975/4096\approx 0.238$ \cite{Kim:2003aa} using new cases of functoriality.  Selberg's conjecture would follow from the general Langland's functoriality conjectures \cite{Langlands:1970aa} concerning symmetric tensor powers of automorphic $L$-functions. See \cite{Shahidi:2004aa} for a survey of the methods leading to such results.



\end{remark}


\begin{remark} The current work grew out of studying a paper by Balslev and Venkov \cite{Balslev:2000aa}. They found that there are no small eigenvalues very close to 1/4 with the corresponding  eigenfunction being odd. Their method was limited by the possible existence of a real zero of the $L$-function attached to a Hecke-Maa\ss{} cusp form. Such a \lq bad\rq{} zero would surely violate the generalized Riemann hypothesis thereby adding credit to the Selberg conjecture (at least for those of us believing GRH). They then used that for $\hbox{GL}_2$ $L$-functions there are no Siegel zeroes  \cite{Hoffstein:1994aa} to conclude that there are no small eigenvalues  close to 1/4. 

In this paper we reprove the relevant  results from  \cite{Balslev:2000aa}, and extend the method to cover the whole interval $]0,1/4[$, i.e. surpassing the problem of real zeros of $L$-functions. Unfortunately Lemma 3 of \cite{Balslev:2000aa} does not seem to be provable as indicated by the authors, nor does it seem to suffice for what they are using it for. Our Conjecture \ref{myconjecture} serves as a valid substitute in the case of Hecke congruence groups, thereby leading only to a conditional proof of the Selberg conjecture.  

 Another new ingredient is adding twists by Dirichlet character allowing us to bypass the problem of the \lq bad\rq{} real zero of the relevant $L$-function as well as other obstacles. 
\end{remark}

We now describe in outline the method of the proof: At the core is the vanishing or non-vanishing at a special point of the following integral 
\begin{equation}\label{PS}
 4\pi i\int_{\G_0(q)\backslash\H}y^2\left(f(z)\frac{\partial \f(z)}{\partial \overline z}+\overline{f(z)}\frac{\partial \f(z)}{\partial z}\right)\overline{E_\infty(z,\overline{s},\chi)}d\mu(z)
\end{equation}
which we denote 
\begin{equation}
  \label{eq:7}
   I(s,\f,\G_0(q),\chi,f)
\end{equation}
Here $f\in M_2^\infty(\G_0(q))$, $\chi$ is an even Dirichlet character, $\f$ is an eigenfunction of the automorphic Laplacian for $(\G_0(q),\chi)$ with eigenvalue $\lambda=s_0(1-s_0)$, $E_\infty(z,s,\chi)$ is the Eisenstein series for the cusp at infinity, and $s$ is a complex number. We want to consider what happens at $s=s_0$. This is done in two different ways. One using asymptotic perturbation theory and one using number theory.

We start by proving the following:
\begin{theorem} \label{PSiszero} If $\lambda=s_0(1-s_0)$ is a small eigenvalue then on Conjecture \ref{myconjecture}
\begin{equation*}
   I(s_0,\f,\G_0(q),\chi,f)=0.
\end{equation*}
\end{theorem}

Our proof of this uses asymptotic perturbation theory. 
Since we are not assuming that $f$ is a cusp form the perturbation of the automorphic Laplacian induced from (\ref{eq:2}) is not necessarily regular, and standard Kato-Rellich theory does not apply. But for eigenvalues which are \emph{stable} Kato's theory of \emph{asymptotic} perturbation \cite[Ch VIII]{Kato:1976aa} still applies. We show that \emph{small} eigenvalues are stable, and we get an asymptotic first order expansion of the eigenprojections and a second order expansion of the eigenvalues. The main challenge in proving stability is the following result:
\begin{theorem}\label{vacation}
  Small eigenvalues move continuously under character deformations (\ref{eq:2}).
\end{theorem}
In Section \ref{pert2} we prove  Theorem \ref{vacation} using the Selberg trace formula with characters.  We note that Theorem \ref{vacation}  (See Corollary \ref{partytime} for a more precise statement) and Selberg's Conjecture \ref{selbergconjecture} almost immediately implies Conjecture \ref{myconjecture}.

 It turns out that we may write  $I(s,\f,\G_0(q),\chi,f)$ as 
\begin{equation}\label{approaching}
  \int_{\G_0(q)\backslash\H}L_\e\f(z)\overline{E_\infty(z,\overline{s},\chi)}d\mu(z)
\end{equation}
where $L_\e$ is the first variation of the Laplacian with respect to the character perturbation (\ref{eq:2}). This integral is usually denoted the Phillip-Sarnak integral. It appeared in \cite{Phillips:1991aa,Phillips:1994aa}, and subsequently in many other papers. On Conjecture \ref{myconjecture} Kato's formula for the first variation of the eigenprojections then allows us to conclude that the Phillips-Sarnak integral is zero for $s=s_0$ where $\f$ has eigenvalue $\lambda=s_0(1-s_0)$ when $0<\lambda<1/4$. We refer to Section \ref{perturbation} for further details.

In Section \ref{number-theory} we use number theoretic methods to prove the following result:
\begin{theorem} \label{PSisnotzero}
Given a primitive Hecke normalized Maa\ss{} cusp form $\f$ related to a small eigenvalue $s_0(1-s_0)$, there exist a primitive Dirichlet character $\psi$ and a weight 2 modular form $f\in M_2(\Gamma_0(q'))$ such that
  \begin{equation*}
     I(s_0,\f\otimes\psi,\G_0(q'),\chi',f)\neq 0.
  \end{equation*}
Here $q'$, $\chi'$ is the level and the nebentypus for the twisted Maa\ss{} form
 $\f\otimes\psi$.
\end{theorem}
Comparing this with Theorem \ref{PSiszero} the existence of a small eigenvalue easily leads to a contradiction proving Theorem \ref{selbergconjecture}.

To prove Theorem  \ref{PSisnotzero} we may arrange that $\f$ is odd (If $\f$ is even the integral is zero identically). We then unfold the integral (\ref{approaching}) and  find that it essentially  equals the Rankin-Selberg convolution of $f$ and $\f$ evaluated at a special point. To be able to handle this convolution we let $f$ be a (linear combination of) weight 2 Eisenstein series. The Fourier coefficients of such a series is a divisor sum and it turns out that the Rankin-Selberg convolution can be expressed in terms of the $L$-functions of $\f$ and $\chi$. Going through the details we find that 
\begin{equation*}
  I(s_0,\f,\G_0(q),\chi,f)=c\frac{\Lambda(s_0+1/2,\f)\Lambda(s_0-1/2,\f)}{\Lambda(2s_0,\chi)}.
\end{equation*} Here $c$ is a simple factor which is clearly non-zero.  The functions $\Lambda(s,\f)$ and $\Lambda(s,\chi)$ are the completed $L$-functions of $\f$ and $\chi$. Since $1/2<s_0<1$ only  $\Lambda(s_0-1/2,\f)$ can be non-zero since the two other functions are evaluated in the domain of absolute convergence. But by twisting by primitive characters we may arrange that this is not the case, using a theorem of Rohrlich.  Going through the argument above with the twisted eigenform $\f\otimes\psi$ we arrive at Theorem \ref{PSisnotzero}. We refer to Section \ref{number-theory} for further details.

We conclude with Section \ref{further} where we give 2 more conjectures which we can prove to be equivalent to Selberg's Conjecture \ref{selbergconjecture}, one of them involving an Eisenstein series twisted by modular symbols introduced by Goldfeld \cite{Goldfeld:1999aa}.

{\bf Acknowledgements:}\newline
It is with great pleasure that I thank Erik Balslev, Roelof Bruggeman, Yiannis N. Petridis, and Alexei B. Venkov for useful discussions in relation to this work. I also thank Jeffrey C. Lagarias, Peter Sarnak, and  Fredrik Str\"omberg  for comments on an earlier version of this paper.

%
%

\section{The automorphic Laplacian and the Selberg trace formula}\label{basics}
We start by reviewing some more or less standard facts and results about the spectrum of the Laplacian. Basic references for this section are  \cite{Selberg:1989aa, Selberg:1956aa,
Venkov:1982aa,Venkov:1990aa,Iwaniec:2002aa}.
\subsection{The automorphic Laplacian}
Let $\H$ be the upper halfplane equipped with the 
hyperbolic measure 
\begin{equation*}
  ds^2=\frac{dx^2+dy^2}{y^2}
\end{equation*}
and corresponding volume form 
\begin{equation*}
  d\mu(z)=\frac{dxdy}{y^2}.
\end{equation*}
Consider a discrete cofinite subgroup $\G$ of  $\psl$. The group $\G$ acts on the upper half plane $\H$ by linear fractional transformations, and we fix  a fundamental domain $\F_\Gamma$ for this action. The action extends to the extended real line $\overline{\R}$ and since $\G$ is cofinite there is a \emph{finite} number $k$ of $\G$-inequivalent fixpoints (cusps) on $\overline{\R}$, which we denote by $\a_i$ $i=1\ldots k.$ (If $\G$ is cocompact the set of cusps is empty. We shall mainly be interested in the non-cocompact case, but since many of our statements are true in the cocompact case also we do not assume anything about cocompactness). The stabilizer of a cusp $\G_{\a_i}$, $i=1\ldots k$, is a maximal parabolic subgroup. This is cyclic with generator $\gamma_{\a_i}$. There exist scaling matrices $\sigma_i\in \psl$ such that
\begin{equation*}
\sigma_i \infty=\a_i, \qquad \sigma_i^{-1}\gamma_{\a_i}\sigma_i=\mattwo{1}{1}{0}{1}.
\end{equation*}
We may assume that the fundamental domain $\F_\Gamma$ is a disjoint union
\begin{equation}\label{funddomain}
  \F_\Gamma=\F_0\cup_{j=1}^k\sigma_j \F_\infty^Y
\end{equation}
where  the closure of $\F_0$ is compact, $Y$ is some fixed number, and \begin{equation*}\F_\infty^Y=\{z\in \H\vert 1/2<\Re(z)\leq 1/2, \Im{z}>Y \}.\end{equation*}
We define the \emph{invariant height} by
\begin{equation}
y_\G(z)=\max_{i}\max_{\g\in\G}\Im (\sigma_i^{-1}\gamma z).
\end{equation} 
This measures how heigh up in the cusps the point $z$ is located. Let $\chi:\Gamma\to S^1$ a multiplicative character i.e. a
 one-dimensional unitary representation. 
Consider the usual Hilbert space of automorphic square integrable functions:
 \begin{equation}
   \label{eq:3}
   \rum{\G,\chi}:=\left\{f:\H\to\C \left\vert \begin{array}{l}f \textrm{ measurable}\\f(\g z)=\chi{(\gamma)}f(z)\end{array}, \, \int_{\F_\G}\abs{f(z)}^2d\mu(z)<\infty \right.\right\} 
 \end{equation}
with the usual inner product
\begin{equation*}
  \inprod{f}{g}=\int_{\F_\G}f(z)\overline{g(z)}d\mu (z)
\end{equation*}

 The \emph{automorphic Laplacian} $A(\G,\chi)$ is a non-negative selfadjoint operator on $ \rum{\G,\chi}$ defined as the closure of $(\Delta, D_{\G,\chi})$ where
\begin{equation}\label{densesubspace}
 \Delta=-y^2\left(\frac{\partial^2}{\partial x^2}+\frac{\partial^2}{\partial y^2}\right)  
\end{equation}
and $D_{\G,\chi}$ consists of smooth function in $\rum{\G,\chi}$ with all derivatives in $x$,$y$  exponentially decaying when $y_\G(z)\to\infty$.
If $\chi(\g_{\a_i})=1$ for some cusp the operator $A(\G,\chi)$ has a continuous spectrum. We call such a cusp \emph{open}, and the remaining we call \emph{closed}. The continuous spectrum may be described in terms of \emph{Eisenstein series}. For $i$ with $\chi(\g_{\a_i})=1$ we define the Eisenstein series
\begin{equation*}
  E_i(z,s,\chi)=\sum_{\g\in\G_i\backslash\G}\overline{\chi(\gamma)}\Im(\sigma_i^{-1}\g z)^s, \qquad \textrm{ for }\Re(s)>1.
\end{equation*} These Eisenstein series has meromorphic continuation to $s\in\C$ and satisfies a functional equation
\begin{equation*}
  E(z,1-s,\chi)=\Phi(s,\chi)E(z,s,\chi).
\end{equation*}
Here $E(z,s,\chi)$ is the vector of Eisenstein series related to open cusps, and $\Phi(s,\chi)$ is called the \emph{scattering matrix}. Its determinant $\phi(s,\chi)=\det\Phi(s,\chi)$ is called the scattering determinant. 
The Eisenstein series satisfies 
\begin{equation}\label{eisenstein-eigenfunction}
\begin{split}
 \Delta E_i(z,s,\chi)=&s(1-s)E_i(z,s,\chi)\\
E_i(\g z,s,\chi)=&\chi(\g)E_i(z,s,\chi), \qquad \g\in \G.
\end{split}
\end{equation}

All poles of Eisenstein series in $\Re(s)>1/2$ are real, simple, and the corresponding residues are eigenfunctions of the automorphic Laplacian $A(\G,\chi)$. We call such eigenfunctions \emph{residual eigenfunctions} and  the corresponding eigenvalues \emph{residual eigenvalues}. Residual eigenvalues lie in the interval $[0,1/4[$.  At a point $1/2<s_0<1$ we have
\begin{equation}\label{resi}
  \inprod{\res_{s_0}E_i(z,s,\chi)}{\res_{s_0}E_j(z,s,\chi)}=\res_{s_0}\Phi_{i,j}(s,\chi)
\end{equation}
which follows from the Maa\ss-Selberg relation (See  \cite[p. 652]{Selberg:1989aa}). In particular we have that 
\begin{equation}
  \norm{\res_{s_0}E_i(z,s,\chi)}^2=\res_{s_0}\Phi_{i,i}(s,\chi)
\end{equation} which shows that $\res_{s_0}\Phi_{i,i}(s,\chi)$ is non-negative. The zero Fourier coefficient (with respect to the cusp $\a_j$) of a residual eigenfunction $\res_{s_0}E_i(z,s,\chi)$ equals 
\begin{equation}
   \int_0^1 \res_{s_0}E_i(\sigma_j z,s,\chi)dx=\res_{s_0}\Phi_{ij}(s,\chi)y^{1-s_0}
\end{equation}

The discrete spectrum 
\begin{equation*}
  0\leq \lambda_0(\chi)\leq \lambda_1(\chi) \leq \lambda_2(\chi)\leq \ldots, \lambda_i(\chi)\to \infty
\end{equation*}
listed according to multiplicity consists of a finite number of residual eigenvalues $s_j(\chi)(1-s_j(\chi))$, where $1/2<s_j(\chi)\leq 1$ is a pole of an Eisenstein series, and a finite or infinite number of \emph{cuspidal eigenvalues}, i.e. eigenvalues such that the corresponding eigenfunction $\varphi$ is cuspidal i.e 
\begin{equation*}
  \int_0^1 \f(\sigma_iz)dx=0
\end{equation*} for all \emph{open} cusps $\a_i$. Such eigenfunctions are called \emph{Maa\ss{} cusp forms}. The continuous spectrum of $A(\G,\chi)$ consists of the set $[1/4,\infty[$ with multiplicity equal to the number of open cusps. The Eisenstein series $E_i(z,1/2+it,\chi)$ are \lq generalized eigenfunctions\rq (or sometimes called \lq eigenpackets\rq). We emphasize that they are not square integrable. The Eisenstein series are orthogonal to Maa\ss{} cusp forms in the following sense: For all $s$ which are not poles of the Eisenstein series
\begin{equation}\label{orthogonal-almost}
\int_{\F_\G}\f(z)\overline{E_i(z,s,\chi)}d\mu(z)=0.
\end{equation}
Since $\f$ is a cusp form the integral is absolutely convergent. Of course since $E_i(z,s,\chi)$ is not square integrable this cannot be interpreted as the usual inner product. 

\subsection{Cuspidal eigenfunctions}
Consider a cuspidal eigenvalue $\lambda=s(1-s)$ ($\Re{s}\geq 1/2$) of $A(\G,\chi)$, and let $S_\lambda(\Gamma,\chi)$ be the set of such cuspidal eigenfunctions. Any $\f\in S_\lambda(\Gamma,\chi)$ is real analytic and admits a Fourier expansion 
\begin{equation}
  \label{fourier-expansion-open}
  \f(\sigma_{i}z)=\sum_{n\neq 0}\rho_\f(n, \a_i)\sqrt{y}K_{s-1/2}(2\pi \abs{n}y)e^{2\pi i nx}
\end{equation}
at open cusps $\a_i$. Here $K_\nu(y)$ is the $K$-Besselfunction of order $\nu$. At closed cusps $\a_i$, where $\chi(\g_{\a_i})=\exp(2\pi i\alpha_{i})\neq 1$, $0<\alpha_i<1$,  $\f$ admits a Fourier expansion 
\begin{equation}
  \label{fourier-expansion-closed}
  \f(\sigma_{i}z)=\sum_{n\in\Z}\rho_\f(n, \a_i)\sqrt{y}K_{s-1/2}(2\pi \abs{n+\alpha_i}y)e^{2\pi i (n+\alpha_i)x}
\end{equation}
The Fourier coefficients satisfy the \lq trivial\rq{}  bound
\begin{equation}
  \label{trivial-bound-on-Fourier}
 \rho_\f(n, \a_i)=O(\sqrt{n}).
\end{equation}
\begin{proposition}\label{gooddecay}
Let  $\f\in S_\lambda(\Gamma,\chi)$. Then $\f$ and all its derivatives decay exponentially at all cusps.
\end{proposition}
\begin{proof}
  This follows easily from (\ref{fourier-expansion-open}),  (\ref{fourier-expansion-closed}), (\ref{trivial-bound-on-Fourier}), the asymptotic behavior of the $K$-Besselfunction
  \begin{equation*}
    K_\nu(y)\sim\left(\frac{\pi}{2y}\right)^{1/2}e^{-y}
  \end{equation*} as $y\to\infty$, and 
  \begin{equation*}
    \left(y^{\nu}K_{\nu}(y)\right)'=-y^{\nu}K_{\nu-1}(y).
  \end{equation*}
\end{proof}
\begin{remark} We note that in particular the above proposition shows that $S_\lambda(\Gamma,\chi)$ is contained in the dense subspace $D_{\G,\chi}$ defining $A(\Gamma,\chi)$. 
\end{remark}

\subsection{The Selberg trace formula}\label{selberg-trace-formula}
The Selberg trace formula relates the spectrum of the automorphic Laplacian with the geometry of the Riemann surface $\GmodH$ through the conjugacy classes of $\G$. To every hyperbolic conjugacy class $\{\g\}_\G$ corresponds a closed geodesic on $\GmodH$ of geodesic length $l(\g)$. We define the norm of $\{\g\}_\G$ as $N(\g)=e^{l(\g)}$, which may be expressed in terms of the trace of $\g$:' 
\begin{equation*}
  N(\gamma)=\exp(2\cosh^{-1}(\tr{\g}/2)).
\end{equation*} Every elliptic conjugacy class $\{\g\}_\G$ is of finite order $m_\g$. Let $\mathcal{P}$ (resp. $\mathcal{R}$) be the set of \emph{primitive} hyperbolic (resp. elliptic) conjugacy classes.  
 Let $h:\C\to\C$ be a test function such that 
\begin{enumerate}
\item $h$ is even
\item $h(r)$ is holomorphic for $\abs{\Im(r)}\leq 1/2+\e$
\item $h(r)=O((1+\abs{r})^{-(2+\e)})$ in the above strip
\end{enumerate}
for some $\e>0$. We shall call such a test function \emph{admissible}.  Let $g$ be the Fourier transform of $h$. For discrete eigenvalues we choose $r_i$ such that $\lambda_i=1/4+r_i^2$. 

Then Selberg found \cite[p. 667]{Selberg:1989aa} the celebrated trace formula 
\begin{gather}
  \label{sporformlen}
\begin{split}
2\sum_{i}h(r_i)&+\frac{1}{2\pi}  \int_\R h(r)\frac{-\phi'}{\phi}\left(\frac{1}{2}+ir,\chi\right)dr  \\
\allowdisplaybreaks =&\frac{\mu(\F_\G)}{2\pi}\int_\R r\tanh{(\pi r)}h(r)dr\\
\allowdisplaybreaks &+\sum_{\{\g\}_\G \in\mathcal{P}}\sum_{k=1}^\infty\frac{\chi(\g^k)2\log N(\g)}{N(\g)^{k/2}-N(\g)^{-k/2}}g(k\log N(\g))\\
\allowdisplaybreaks  &+\sum_{\{\g\}_\G \in\mathcal{R}}\sum_{1 \leq \nu<m_\g}\frac{2\chi(\g^\nu)}{m_\g\sin(\pi\nu/m_\g)}\int_\R h(r)\frac{e^{-{\pi \nu r}/{m_\g}}}{1+e^{-2\pi r}}dr\\
\allowdisplaybreaks &-2\left(\sum_{\chi(\g_{\a_i})=1}\log2+\sum_{\chi(\g_{\a_i})\neq 1}\log\abs{1-\chi(\g_{\a_i})}\right)g(0)\\
\allowdisplaybreaks  &+\frac{1}{2}\tr{I-\Phi(1/2,\chi)}h(0)
-\frac{k_1}{\pi}\int_{\R}h(r)\frac{\Gamma'}{\Gamma}(1+ir)dr
\end{split}
\end{gather}
where in the last line   $\Gamma$ denotes Eulers Gamma function (not to be confused with the discrete group $\G$), and $k_1$ equals the number of open cusps.

\subsection{The resolvent}
We denote by $R(s)=(A(\G,\chi)-s(1-s))^{-1}$ the resolvent of $A(\G,\chi)$ defined on the resolvent set $\rho(A(\G,\chi))$ which is the complement of the spectrum $\sigma(A(\G,\chi))$. We remind that the resolvent is a bounded $L^2$-operator which satisfies 
\begin{equation}\label{resolventbound}
   \norm{R(s)}\leq \frac{1}{dist(s(1-s),\sigma(A(\G,\chi)) )}
\end{equation}
If $s_0(1-s_0)$ is an \emph{isolated} eigenvalue the reduced resolvent for  $s_0(1-s_0)$  is defined as 
\begin{equation*}
  R_0(s)=R(s)(1-P)
\end{equation*}
where $P$ denotes the projection to the $s_0(1-s_0)$-eigenspace. This extends holomorphically across $s=s_0$ (See \cite[Ch III \S 6.5]{Kato:1976aa}).

We need some control of how the image of the (reduced) resolvent grows at the cusps. In order to keep control of such behavior Faddeev \cite{Faddeev:1967aa} introduced certain Banach spaces: 
Every $f:\F_\G\to\C$ may be written in terms of its $k+1$ components:
\begin{align*}
  f_0(z)=&f(z) &z\in \F_0, &\\
  f_j(z)=&f(\sigma_j z) &z\in \F_\infty^Y, &\qquad j=1\ldots k.
\end{align*}
For $\mu\in \R$ define $B_\mu$ as those $f$ for which all its components are continuous on $\F_0$,  $\F_\infty^Y$ respectively and such that  $f_j(z)/y^\mu$ is bounded on $\F_\infty^Y$, $j=1\ldots k$. The $\mu$-norm is defined as
\begin{equation*}
  \norm{f}_\mu=\max_{z\in\F_0}{\abs{f_0(z)}}+\sum_{j=1}^k\max_{z\in\F_\infty^Y}\abs{\frac{{f_j(z)}}{y^\mu}}.
\end{equation*}
This norm turns $B_\mu$ into a Banach space, and every function in $B_\mu$ grows at most like $y^\mu$ at all cusps.
\begin{proposition}\label{faddeev-result}Let $\lambda_0=s_0(1-s_0)<1/4$ be a small eigenvalue of $A(\Gamma,\chi)$, and let $R_0(s_{0})$ be the reduced resolvent at $\lambda_0$. Then  $R_0(s_{0})$ maps $B_0$ to  $B_{1-s_0+\e}$ for all $\e>0$.\end{proposition}
\begin{proof}
This follows from Faddeev's theory \cite{Faddeev:1967aa}. We refer to Lang \cite[XIV \S 11]{Lang:1985aa}. More precisely - in p. 334 l. 7 choose $\mu=1/2+\delta$. This choice of $\mu$ will not allow us to be deal with embedded eigenvalues as in \cite{Lang:1985aa,Faddeev:1967aa} but we get better bounds for small eigenvalues. Following the argument in \cite{Lang:1985aa} to p. 338 l. 4 we find that when $1/2+\delta<\Re(s)<2$ and $s(1-s)$ not an eigenvalue (except possibly $s(1-s)=\lambda_0$ which is the point we care about).
\begin{equation*}
R_0(s)=  R(s)-\frac{P}{\lambda_0-s(1-s)}:B_0\to B_{1/2-\delta}
\end{equation*} where $P$ is the projection to the $\lambda_0$-eigenspace. By choosing $\delta$  such that  $1/2+\delta+\e=s_0$ where we assume that $\e$ is small enough that $\delta>0$ which is possible since $1/2<s_0<1$.  Therefore $\Re(s_0)> 1/2+\delta$ and the desired result follows. Lang \cite{Lang:1985aa} only considers $\chi=1$ but Venkov \cite{Venkov:1982aa} explains the relatively small modifications needed to deal with the case $\chi\neq 1$. 
  
\end{proof}

\section{Small eigenvalues and character perturbation}\label{perturbation}
In this section we describe how holomorphic forms of weight two for $\G$ induces a perturbed Laplacian. 
We will then explain how Kato's asymptotic perturbation  theory \cite[Chapter VIII]{Kato:1976aa} may be applied to study small eigenvalues. We insist on not assuming that the weight 2 form is cuspidal, which is why we are using  Kato's asymptotic perturbation theory instead of the more standard  regular perturbation theory which does not seem to apply. We are still working in the general situation of section \ref{basics} i.e. a general discrete cofinite subgroup of $\psl$.

\subsection{The perturbed automorphic Laplacian}\label{erc-application-is-over}
Let $f\in M_2(\G)$ be a holomorphic form of weight 2, not necessarily a cusp form. Then
\begin{equation*}
  \alpha(z)=\Re(f(z)dz)=\frac{f(z)}{2}dz+\frac{\overline{f(z)}}{2}d\overline{z}
\end{equation*}
is a harmonic $\G$-invariant 1-form, and for $\e\in \R$ the map

 \begin{equation}\begin{array}{llccc}\chi(\e,\alpha)&:&\G&\to& S^1\\
 &&\g&\mapsto&\displaystyle e^{2\pi i\e\int_{z_0}^{\g z_0}\alpha}
 \end{array}.
 \end{equation} 
is a multiplicative character of $\G$. Here $z_0\in \H$ may be chosen arbitrarily as the integral does not depend on the specific choice.

For a given character $\chi$ of $\G$ and $f\in M_2(\G)$ we form the character
\begin{equation*}
  \chi_\e:=\chi\cdot\chi(\e,\alpha),
\end{equation*}
 and we want to investigate how isolated eigenvalues $\lambda_i(\e)$ of the Laplacians $L(\G,\chi_\e)$ behave in the limit as $\e$ approaches zero. To make the problem more susceptible to the standard techniques of perturbation theory we consider a unitarily equivalent operator as in \cite{Phillips:1991aa, Phillips:1994aa}. Consider the unitary operator
\begin{equation}\begin{array}{llccc}U(\e)&:&\rum{\G,\chi_0}&\to& \rum{\G,\chi_\e}\\
 &&f(z) &\mapsto&\displaystyle \exp\left(2\pi i\e\int_{z_0}^{ z}\alpha\right)f(z).
 \end{array}
 \end{equation} \label{perturbed-laplacian}
The perturbed Laplacian 
\begin{equation}
\begin{CD} 
\rum{\G,\chi_\e} @>A(\G,\chi_\e)>>\rum{\G,\chi_\e}\\ 
@VU(\e)^{-1}VV @VVU(\e)^{-1}V\\ 
\rum{\G,\chi_0} @>L(\e)>> \rum{\G,\chi_0} 
\end{CD} 
\end{equation}
now becomes
\begin{equation} \label{perturbed}
  L(\e ):=U(\e)^{-1}A(\G,\chi_\e)U(\e)=A(\G,0)+\e L_\e+\e^2  L_{\e\e} 
\end{equation}
 where $L_\e$ resp.  $L_{\e\e}$ are the selfadjoint operators obtained as the closure 
 \begin{equation}\label{leavingsoon}
   h\mapsto 4\pi i \inprod{dh}{\alpha}, \qquad h\mapsto -4\pi^2\inprod{\alpha}{\alpha}h 
 \end{equation}
defined on $D_{\G,\chi_0}$ (See (\ref{densesubspace})).
Here 
\begin{equation*}
  \inprod{f_1dz+f_2d\overline z}{g_1dz+g_2d\overline z}=2y^2(f_1\overline{g_1}+f_2\overline{g_2})
\end{equation*}
We will call this kind of perturbations \emph{character perturbations of the automorphic Laplacian $A(\Gamma,\chi)$}. 

In the calculation leading to (\ref{perturbed}) we have used that $\alpha$ is harmonic. For more general 1-forms there is an extra term (See \cite[(2.5)]{Phillips:1994aa}).
We emphasize that the identity in (\ref{perturbed-laplacian}) should be understood in the following way: The operator $L(\e)$ is defined as the closure of the operator defined by $A(\G,0)+\e L_\e+\e^2  L_{\e\e}$ on the dense subspace $U(\e)D_{\G,\chi_\e}=D_{\G,\chi_0}$


\begin{remark}\label{discontinuous}
  The upshot of conjugating to the fixed space $D_{\G,\chi_0}\rum{\G,\chi_0}$ is that standard methods from perturbation theory may be applied. If $\chi_\e$ for small $\e$ has the same cusps open and closed as $\chi_0$ -- as happens  if $f\in M_2(\G)$ is cuspidal i.e.  $f\in S_2(\G)$, or if $\G$ is cocompact -- then the perturbation is analytic (See \cite[Prop 2.8]{Epstein:1987aa}), and there is a nice perturbation theory for isolated eigenvalues (See \cite[Chapter VII]{Kato:1976aa}). But we want to be able to handle non-cuspidal $f\in M_2(\G)$ and in this situation \emph{the number of open cusps may go down} when $\e\neq 0$. A first indication that such a perturbation requires extra care comes from observing that the two last lines in (\ref{sporformlen}) are \emph{not} continuous in the limit $\e\to 0$. In spite of this perturbation theory has strong results to offer:  See sections \ref{pert1} and  \ref{pert2}.

\end{remark}

 \subsection{Kato's asymptotic perturbation theory} \label{pert1}

Perturbation of the smallest eigenvalue has been studied by Phillips and Sarnak \cite{Phillips:1987aa}, Epstein \cite{Epstein:1987aa}, Petridis and Risager \cite{Petridis:2007aa} and others to count the number of closed geodesics on Riemann surfaces. Perturbations of embedded eigenvalues has been investigated by Phillips and Sarnak \cite{Phillips:1991aa, Phillips:1994aa}, Wolpert \cite{Wolpert:1994aa}, Balslev and Venkov \cite{Balslev:2001aa, Balslev:2005aa} and others to explore whether a Weyl law holds for the generic finite volume Riemann surface (I.e. is the Roelcke-Selberg conjecture true?).

 We now cite the relevant parts of \cite[Chapter VIII]{Kato:1976aa} in a slightly generalized form. We refer to \cite{Kato:1976aa} for further explanations and motivation.

Let $\{T_\e\}_{\e>0 }$ be a set of closed operators in a Banach space $X$. Let $R_\e(\zeta)=(T_\e-\zeta)^{-1}$ be the resolvent defined for $\zeta\in \rho(T_\e)$ the resolvent set.  The \emph{region of boundedness for $\{R_\e(\zeta)\}$} is the set  $\Delta_b$  consisting of $\zeta\in \C$ such that for $\e_0>0$ sufficiently small $\{\norm{R_\e(\zeta)}\}_{0<\e\leq\e_0 }$ is bounded. Define the  \emph{region of strong convergence for  $\{R_\e(\zeta)\}$} as the set  $\Delta_s$ of $\zeta\in\C$ such that the strong limit $\slim_{\e\to 0} R_\e(\zeta)=R'(\zeta)$ exists. Assume that $\Delta_s$ is nonempty. Then either $R'(\zeta)$ is invertible for no $\zeta\in\Delta_s$ or $R'(\zeta)$ equals the resolvent $R(\zeta)$ of a unique closed operator $T$ on X. In the latter case $\Delta_s=\rho(T)\cap \Delta_b$, and we say \emph{that $T_{\e} $ converges strongly to $T$  in the generalized sense.} Hence convergence in the generalized sense means that $R_\e(\zeta)$ converges strongly to $R(\zeta)$ when  $\zeta\in\Delta_s$. We remind that a core for a closed operator $T$ is the domain of any closable operator $S$ such that the closure of $S$ equals $T$. We have the following criterion for generalized strong convergence:

\begin{theorem}\cite[p. 429]{Kato:1976aa}\label{coreconvergence}
Let $T_\e$, $T$ be closed operators on $X$. Assume that there is a core $D$ of $T$ such that each $u\in D$ belongs to $D_{T_\e}$ for $\e$ sufficiently small and $T_\e u\to Tu$ as $\e\to 0$. If $\rho(T)\cap \Delta_b\neq \emptyset$ the $T_\e$ converges strongly to $T$ in the generalized sense and $\Delta_s=\rho(T)\cap \Delta_b$.
\end{theorem}

Assume that $T_\e$ converges strongly to $T$ in the generalized sense. Unfortunately, in  general not much can be said about the spectrum of $T_\e$ ($\e$ small) close to $\lambda$ even under the assumption that $\lambda$ is an isolated eigenvalue of $T$ (See \cite[Ch. VIII, \S 1. 4]{Kato:1976aa}). We therefore introduce the notion of a \emph{stable eigenvalue} for which we will be able to say more. 

Still assuming that $T_\e$ converges strongly to $T$ in the generalized sense we say that an isolated finite multiplicity eigenvalue $\lambda$ of $T$  is \emph{stable} (in the sense of Kato)  if
\begin{enumerate}
\item $\Delta_s$ contains a deleted neighborhood of $\lambda$.\\
 In particular,  for some $\delta>0$,  $\{0<\abs{\zeta-\lambda}= \delta\}\subset \Delta_s$. From \cite[Theorem 1.2]{Kato:1976aa} the convergence $R_\e(\zeta)\to R(\zeta)$ is uniform on $\G=\{\zeta\in\C\vert\, \abs{\zeta-\lambda}=\delta\}$. We can therefore define the projection 
  \begin{equation*}
    P_\e=-\frac{1}{2\pi i}\int_{\G}R_\e(\zeta)d\zeta
  \end{equation*}
which projects to the eigenspaces with $\abs{\zeta-\lambda}<\delta$.  The projections $P_\e$ converges strongly to 
\begin{equation*}
  P=-\frac{1}{2\pi i}\int_{\G}R(\zeta)d\zeta
\end{equation*}
as $\e\to 0$.
\item $\dim P_\e \leq \dim P$ for sufficiently small $\e$
\end{enumerate}

We will now consider a particular type of family of operators. Let $\mathscr{C}(X)$ be the set of closed operators on $X$. Assume that we have operators $T$, $T^{(1)}$,  $T^{(2)}$, and a parameter $\e$ are given such that
\begin{enumerate}
\item \label{farfar} $T\in \mathscr{C}(X)$
\item $D=D_T\cap D_{T^{(1)}}\cap D_{T^{(2)}}$ is a core of $T$
\item $T(\e)\in \mathscr{C}(X)$ and for $0<\e\leq 1$ it is an extension of the operator $T+\e T^{(1)}+\e^2 T^{(2)}$ defined with domain $D$.
\item \label{notempty} $\rho(T)\cap\Delta_b\neq \emptyset$ where $\Delta_b$ is the region of boundedness for the family $(T(\e)-\zeta)^{-1}$.
\end{enumerate}
\begin{remark}\label{faddeevishere}
We note that by Theorem \ref{coreconvergence} we can conclude that $T(\e)$ converges strongly to $T$ in the generalized sense, and $\Delta_s=\rho(T)\cap\Delta_b$.
\end{remark}

Kato proves the following:  
\begin{theorem} \label{kato-bigtheorem}
  Let $X$ be a Hilbert space, and assume that we have a family as above with $T$, $T(\e)$ selfadjoint, and $T^{(1)}$ symmetric. 
Let $\lambda$ be a stable 
eigenvalue of $T$ of dimension $m<\infty$ with projection $P$ and assume $PX\subset D$. Then the $m$ eigenvalues for $T(\e)$ close to $\lambda$ may be numbered in the form $\mu_{jk}(\e)$, $j=1,\ldots,s$, $k=1,\ldots,m_j$ such that they have an asymptotic expansion
\begin{equation}
  \label{eigenvalueexpansion}
\mu_{jk}(\e)=\lambda+\e\mu_j^{(1)}+\e^2\mu_{jk}^{(2)}+o(\e^2) \textrm{ as } \e\to 0_+
\end{equation}
Here $\mu_j^{(1)}$ are the eigenvalues of $PT^{(1)}P$, with corresponding eigenprojections $P_j^{1}$.
The total projection $P_j(\e)$ for the $m_j$ eigenvalues $\mu_{jk}(\e)$ has an asymptotic expansion
\begin{equation}
  \label{projectionexpansion}
P_j(\e)=P_j^{1}+\e P_j^{11}+o(\e)_s  
\end{equation}
where $o(\e)_s $ denotes an operator such that $\e^{-1}o(\e)$ converges strongly to $0$ as $\e\to 0_+$, and 
\begin{equation}
  \label{second-term-in-projection}
  P_j^{11}=-S_\lambda T^{(1)}P_j^{1}+\sum_{i}P_j^{1}A_{ij}.
\end{equation}
Here $A_{ij}$ are bounded operators, and $S_\lambda$ is the reduced resolvent of T at $\lambda$.
\end{theorem}
\begin{remark}
This is a simplified version of Kato's Theorem 2.9, Remark 2.10, and footnote 2 on p. 449. Kato assumes $T^{(2)}=0$. Wolpert \cite[263-266]{Wolpert:1994aa} describes which minor changes are needed to allow $T^{(2)}\neq 0$. Note that (0) p.264 in \cite{Wolpert:1994aa} is trivially satisfied in the version of the problem that we stated.   The explicit form of the bounded operators $A_{ij}$ is given in Kato's Theorem 2.9 but we do not need it.
\end{remark}

\subsection{Stability of small eigenvalues under character perturbation} \label{pert2} In this section we will show that small eigenvalues move continuously under character perturbations with $f\in M_2(\Gamma)$, i.e. in the setup described in Section \ref{erc-application-is-over}. This will be a crucial ingredient when we will show later that small eigenvalues are \emph{stable} in the sense of Kato. The proof of Lemma \ref{small-eigenvalues} below is strongly inspired by \cite{Huntley:1995aa, Huntley:1997aa}  where the authors prove similar result in the case of pinching geodesics, using the heat kernel. Similar results where obtained by Hejhal (pinching geodesics) \cite[Thm 7.2 ]{Hejhal:1990aa}  and Venkov (regular deformations) \cite[Thm 7.1.1]{Venkov:1982aa} using the resolvent kernels. 

We recall a few standard facts about the Laplace transform.
For a (sufficiently nice) function $f$ on $\R_+$  we define its Laplace transform as  
\begin{equation}
  \label{laplace-transform}
  \laplace(f)(z)=\int_0^\infty e^{-zt}f(t)dt.
\end{equation}
If for instance $f$ is piecewise continuous and  real-valued satisfying $\abs{f(t)}\leq M e^{ct}$ then  $\laplace(f)(z)$ exist for all complex $z$ in a half-plane $\Re(z)>a_0$. The inverse transform is given by
\begin{equation*}
f(u) =  \frac{1}{2\pi i}\int_{a-i\infty}^{a+i\infty}e^{zu}\laplace(f)(z)dz.
\end{equation*}
which hold for any $a>a_0$. 
Define for $\rho>0$ \begin{equation*}f_\rho(t)=\int_0^t\frac{(t-u)^{\rho-1}}{\Gamma(\rho)}f(u)du.\end{equation*}
We recall  (\cite[Thm 8.1]{Widder:1941aa}) that  when $a>0$, $a>a_0$
\begin{equation}\label{god-invers}
 \frac{1}{2\pi i}\int_{a-i\infty}^{a+i\infty}e^{zu}\frac{\laplace(f)(z)}{z^\rho}dz=\begin{cases}f_\rho(u) & u\geq 0\\ 0 &u< 0\end{cases} 
\end{equation}

We now define the spectral counting function 
\begin{equation*}
  N(T,\e)=\#\{\l_i(\chi_\e)\leq T\}, 
\end{equation*} i.e. $N(T,\e)$ counts the number of  eigenvalues for the Laplacian $L(\G,\chi_\e)$ which are less than $T$. 
\begin{lemma}\label{small-eigenvalues} Fix $T<1/4$ which is not an eigenvalue for $L(0)$. Then the spectral counting function $N(T,\e)$ is continuous at $\e=0$.
\end{lemma}

From this we readily deduce the following corollary:

\begin{corollary}\label{partytime}
 Let $\lambda$ be any number such that $0<\lambda<1/4$. For sufficiently small $\e$ the number of eigenvalues of $L(\e)$ which are close to $\lambda$ (counted with multiplicity) is equal to the multiplicity with which $\lambda$ is an eigenvalue for $L(0).$ More precisely 
\begin{equation}\lim_{\e\to 0}(N(\l+\delta,\e)-N(\l-\delta,\e))=\lim_{t\to 0_+}N_0(\l+t,0)-N_0(\l-t,0)\end{equation}
for sufficiently small $\delta>0$.  
\end{corollary}
We now prove Lemma \ref{small-eigenvalues}:
\begin{proof}
  Consider the Selberg trace formula (\ref{sporformlen}) with the test function $h(r)=\exp(-zr^2)$ where $z$ is complex with $\Re(z)>0$. This is certainly admissible. Then $g(x)=(4\pi z)^{-1/2}\exp(-x^2/4z)$. We multiply the resulting trace formula by $\exp(-z/4)$ giving an identity of the form 
  \begin{equation}\label{identity-1}
    \sum e^{-z\lambda_i(\chi_\e)}+\ldots
  \end{equation}

Let $f$ be sufficiently nice, e.g $f(u)=u^{w-1}$, $w-1\geq 0$.
 Let $T<1/4$ and multiply (\ref{identity-1}) by $\laplace{f}(z)e^{zT}/z$. By using (\ref{god-invers}) with $\Re(z)=a>0$ it is straightforward to derive that
\begin{equation}\label{continuous-equation}
 \sum_{\lambda_n(\chi_\e)\leq T}f_1(T-\lambda_i(\chi_\e)) = \sum_{\{\g\}_\G \in\mathcal{P}}\sum_{k=1}^\infty\frac{\chi_\e(\g^k)2\log N(\g)}{N(\g)^{k/2}-N(\g)^{-k/2}}v_T(k\log N(\g))
\end{equation}
where
\begin{equation*}
  v_T(x)=\frac{1}{2\pi i}\int_{a-i\infty}^{a+i\infty}\frac{\exp(-z/4-x^2/4z)}{\sqrt{4\pi z}}\frac{\laplace(f)(z)e^{zT}}{z}dz
\end{equation*}
We note that $T<1/4$ is crucial. If $T>1/4$ there will be more terms. 
The above operation removes discontinuous terms in the trace formulae (See remark \ref{discontinuous})
By using dominated convergence we easily see that the left hand side of (\ref{continuous-equation}) is continuous in $\e$. Hence, as long as $T<1/4$ and $w\geq 1$
\begin{equation}\label{easter}
N_w(T,\e):= \sum_{\lambda_n(\chi_\e)\leq T}(T-\lambda_i(\chi_\e))^w 
\end{equation}
is continuous in $\e$ at $\e=0$. The theorem is complete if we can verify that when $T<1/4$ not an eigenvalue for $\e=0$  $N_0(T,\e)$ is continuous in $\e$ at $\e=0$. This follows from the above by an approximation argument: 

By an elementary consideration, using the mean value theorem and that  $N_w(T,\e)$ is monotonically increasing in $T$ 
\begin{equation}\label{meanvalue}
  N_0(T,\e)\leq \frac{ N_1(T+\delta,\e) - N_1(T,\e)}{\delta}\leq N_0(T+\delta,\e)
\end{equation}
when $0<\delta$ and $\delta+T<1/4$.
Letting $\e\to 0$ and using that $N_1(T,\e)$ is continuous when $\e\to 0$  the first inequality in (\ref{meanvalue}) gives 
\begin{equation}\label{soon}
  \limsup_{\e\to 0} N_0(T,\e)\leq \frac{ N_1(T+\delta,0) - N_1(T,0)}{\delta}
\end{equation}
Equation (\ref{meanvalue}) gives also
\begin{equation*}
  \frac{ N_1(T,\e) - N_1(T-\delta,\e)}{\delta}\leq N_0(T,\e)
\end{equation*}
which by considering $\e\to 0$ gives
\begin{equation}\label{iwill}
   \frac{ N_1(T,0) - N_1(T-\delta,0)}{\delta}\leq \liminf_{\e\to 0}N_0(T,\e)
\end{equation}
We then use that since $T$  is not an eigenvalue when $\e=0$
\begin{equation*}
 N_0(T,0)=\lim_{\delta\to 0} \frac{ N_1(T+\delta,0) - N_1(T,0)}{\delta}
\end{equation*}
 to conclude from (\ref{soon}) and (\ref{iwill}) that 
 \begin{equation*}
    \limsup_{\e\to 0} N_0(T,\e)\leq  N_0(T,0) \leq \liminf_{\e\to 0}N_0(T,\e).
 \end{equation*}
It  follows that $N_0(T,\e)$ is continuous at $\e=0$ if  $T$  is not an eigenvalue when $\e=0$.
\end{proof}

\begin{lemma}  \label{region-of-strong-convergence} Assume that $\chi_0$ leaves some cusps open. The region of strong convergence for $L(\e)$ equals $\rho(L(0))$ -- the resolvent set of the automorphic Laplacian $A(\G,\chi_0)$ on $\rum{\Gamma,\chi_0}$ -- with 0 as the only possible exception. 
\end{lemma}
\begin{proof}
We show that $\Delta_b$ contains  $\rho(L(0))\backslash\{0\}$. The result will then follow from Remark \ref{faddeevishere}. Consider $\zeta\in \rho(L(0)) \backslash\{0\}$. If $\Im (\zeta)\neq 0$ then it follows from (\ref{resolventbound}) and the fact that $L(\e)$ is selfadjoint that  $\norm{R(\zeta,\e)}\leq 1/{\Im(\zeta)}$ which shows that $\zeta\in \Delta_b$

If $\zeta\in \R$ then $\zeta<1/4$  since we assume that $\chi_0$ leaves a cusp open i.e. there is continuous spectrum in $[1/4,\infty[$. If $\zeta<0$ then the result follows from (\ref{resolventbound}) and non-negativity of $L(\e)$. If $0<\zeta<1/4$ it follows from Corollary \ref{partytime} that for $\e$ small enough there are no eigenvalues of $L(\e)$ in a small neighborhood of $\zeta$ and it then follows from (\ref{resolventbound}) that $\zeta\in\Delta_b$.
 \end{proof}
\begin{remark}
We note that without the assumption of $\chi_0$ in Lemma \ref{region-of-strong-convergence} we can still conclude (with the same proof) that the region of strong convergence \emph{contains}  $\rho(L(0))\backslash (\{ 0 \}\cup [1/4,\infty[)$. But in this case the perturbation is analytic and much stronger results than what we are obtaining are available. \end{remark}

\begin{theorem}\label{stable} The operator $L(\e)$ converges strongly to $L(0)$ in the generalized sense. Let $\lambda<1/4$ be a small eigenvalue for the automorphic Laplacian $A(\Gamma, \chi_0)$ on $\rum{\Gamma,\chi_0}$. Then $\lambda$ is stable in the sense of Kato under character perturbation.
\end{theorem}
\begin{proof} The operator $L(\e)$ fits in the framework of (\ref{farfar})-(\ref{notempty}) before Remark \ref{faddeevishere}  so by Remark \ref{faddeevishere} the family of operators $L(\e)$, $0<\e\leq 1$,  converges strongly to $L(0)$ in the generalized sense.
It follows from Lemma \ref{region-of-strong-convergence} that $\Delta_s$ contains a deleted neighborhood of $\lambda$ and it follows from Corollary \ref{partytime} that $\dim P=\dim P_\e$. We conclude that $\lambda$ is stable in the sense of Kato.  
\end{proof}

\subsection{Families of eigenfunctions, and vanishing of the Phillips-Sarnak condition} We concluded the preceding section with Theorem \ref{stable} which allows us to apply Kato's Theorem \ref{kato-bigtheorem} to character perturbations of small eigenvalues.  In this section we study the ramifications of this for the Phillips-Sarnak integral.
From this point of most of our considerations makes sense only if we  assume that the group $\G$ has parabolic elements (i.e. $\GmodH$ has cusps). So from now on we assume that $i\infty$ is a cusp and that $f$ is cuspidal at infinity, i.e.
\begin{equation*}
  \int_0^1f(z)dx=0.
\end{equation*}i.e. in the notation from the introduction we are assuming that $f\in M_2^\infty(\G)$. We will construct such functions in section \ref{formsconstructed}.

Theorem \ref{kato-bigtheorem} and Theorem \ref{stable} implies that for an eigenvalue $0<\lambda=s_0(1-s_0)<1/4$ with eigenfunction $\f(z)$ there exist a family of eigenfunctions
\begin{equation*}
  \e\mapsto \f(z,\e)
\end{equation*} 
$0\leq \e<\e_0$, $\f(z,\e)\in \rum{\G}$ of with 
\begin{equation*}
  (L(\e)-\lambda(\e))\f(z,\e)=0, \qquad \f(z,0)=\f(z)
\end{equation*}
such that  $\f(z,\e)$ is right-differentiable in $L^2(\G,\chi_0)$) at $\e=0$, and $\lambda(\e)$ is right-differentiable at $\e=0$. By the discussion in section \ref{basics} the pull-back $\hat\f(z,\e):=U(\e)\f(z,\e)$ is (for each fixed $\e$) a linear combination of a cuspidal eigenfunction of $A(\G,\chi_\e)$ and residual eigenfunctions.

We need some control over how the first variation $\f'(z,s)$ of $\f(z,\e)$ behaves at $y\to\infty$. Faddeev's theory (Proposition \ref{faddeev-result}) gives $O(y^{1-s_0+\varepsilon})$ but we need something slightly better. In fact if we knew that the first variation satisfies $O(y^{1-s_0-\varepsilon})$ we could prove Selberg's conjecture unconditionally. To circumvent this we make one more assumption. We write $E_i(z,s,\e):=E_i(z,s,\chi_\e)$.
\begin{assumption}\label{goaway} We assume that for $\delta>0$ sufficiently small there exist $\e_0>0$ such that $E_{\infty}(z,s,\e)$ is regular when $\abs{s-s_0}<\delta$ and $\abs{\e}\leq\e_0$. \end{assumption}
This assumption clearly implies that  $\hat\f(z,\e)$ does not have a
residual component in the direction of the cusp at infinity. We notice
that this is really an assumption on the group $\Gamma$, and that we
are still not making any assumption about the group being
arithmetic; only that the corresponding Eisenstein series doesn't
develop a pole close to $s_0$ when the character perturbation is \lq
turned on \rq.

 \begin{theorem} \label{integraliszero}
 Let $\f$ be an eigenfunction related to a \emph{small} eigenvalue $\lambda=s_0(1-s_0)<1/4$ for the Laplacian on $\rum{\G,\chi}$. On Assumption \ref{goaway} we have  
 \begin{equation}\label{phillips-sarnak}
   \int_{\F_\G}L_\e\f(z)\overline{E_\infty(z,s_0,0)}d\mu(z)=0.
 \end{equation}
 \end{theorem}
  
\begin{proof}
We may apply Theorem \ref{kato-bigtheorem} with $T=A(\G,\chi_0)=L(0)$, $T^{(1)}=L_\e$, $T^{(2)}=L_{\e,\e}$, and $T(\e)=L(\e)$. We notice that $D=D_T\cap D_{T^{(1)}} \cap D_{T^{(2)}}$ contains $D_{\G,\chi_\e}$ so $P\rum{\Gamma,\chi_0}\subseteq D$ by Proposition \ref{gooddecay}. The eigenvalue $\lambda$ is stable by Theorem \ref{stable}. By linearity of the integral (\ref{phillips-sarnak}) it is enough to prove the claim for eigenfunctions such that $P_j^{1}\f=\f$

We set $\f(z,\e)=P_j(\e)\f$, with corresponding eigenfunction $\lambda(\e)=s_0(\e)(1-s_0(\e))$. By applying $(L(0)-\lambda)$ to (\ref{second-term-in-projection}) we find
\begin{equation}\label{intermediate}
  (L(0)-\lambda)  P_j^{11}=-(L(0)-\lambda) S_\lambda L_\e P_j^{1}
\end{equation}
Using $(L(0)-\lambda)S_\lambda=(1-P)$, $S_\lambda P=0$ (\cite[Ch III, (6.34)]{Kato:1976aa}), and $P(L_\e-\mu^1_{j})P_j=0$ (by the definition of $\mu^1_{j}$) we find 
\begin{align*}
  (L(0)-\lambda)  P_j^{11}=& -(L(0)-\lambda) S_\lambda (L_\e-\mu^1_{j}) P_j^{1}\\
=&-(1-P)(L_\e-\mu^1_{j}) P_j^{1}=-(L_\e-\mu^1_{j}) P_j^{1}
\end{align*}
Therefore,  by using that $\inprod{h}{E_\infty(z,s,0)}=0$ for \emph{any} cusp form $h$ and any $s$ not a pole of $E_\infty(z,s,0)$, we find that (\ref{phillips-sarnak}) equals
\begin{equation*}
  \int_{\F_\G}(L_\e-\mu^1_{j}) P_j^{1}\f \overline{E_\infty(z,s_0,0)}d\mu(z)=
- \int_{\F_\G}(L(0)-\lambda)  P_j^{11}\f \overline{ E_\infty(z,s_0,0)}d\mu(z).
\end{equation*}
Using (\ref{intermediate}) again this equals 
\begin{equation}\label{ceg}
  \int_{\F_\G}(L(0)-\lambda)  S_\lambda L_\e P_j^{1}\f  \overline{E_\infty(z,s_0,0)}d\mu(z).
\end{equation}

This looks like an inner product between  $(L(0)-\lambda)  S_\lambda L_\e P_j^{1}\f$ and $E_\infty(z,s_0)$. We would like to use the selfadjointness of $L(0)$ to move $L(0)-\lambda$ to the other side and then use $(\Delta+s(1-s))E_\infty(z,s)=0$ to conclude that (\ref{ceg}) equals zero. This is of course not rigorous since $E_\infty(z,s_0)$ is not square integrable, but we can make the same idea work using an approximation argument \emph{if we can control the growth of} $S_\lambda L_\e P_j^{1}\f$ as $y\to \infty$. We can do that on Assumption \ref{goaway}.

 We claim that $ S_\lambda L_\e P_j^{1}\f$ decays exponentially as $y\to\infty$. Using (\ref{second-term-in-projection}) and (\ref{projectionexpansion}) we see that this follows if we can prove that $\hat\f(z,\e)=U(\e)\f(z,\e)$ has zero Fourier coefficient at infinity equal to zero for $\e$ sufficiently small, i.e. in the expansion
 \begin{equation*}
   \hat\f(\sigma_iz,\e)=a_i(\e)y^{s_0(\e)}+\sum_{m\neq 0}a_i(m,\e)K_{s_0(\e)-1/2}(2\pi \abs{m} y)e(my)
 \end{equation*} we have $a_\infty(\e)=0$ for $\e$ sufficiently small.
To see this we use the (generalized) Maa\ss-Selberg relation (See \cite{Iwaniec:2002aa} Theorem 6.14) which in our case gives that for $1/2<s<1$, $s\neq s_0(\e)$ we have 
\begin{align*}
\inprod{ \hat\f^Y(z,\e)}{\tilde E_j^Y(z,s,\e)}=&\frac{1}{s_0(\e)-s}\sum_{i}-a_i(\e)\delta_{ij}Y^{s_0(\e)-s}\\
&+\frac{1}{s_0(\e)+s-1}\sum_{i}-a_i(\e)\overline{\Phi_{ij}(s,\e)}Y^{1-s_0(\e)-s}
\end{align*}
when $Y$ is sufficiently large. The sum is over all open cusps and 
\begin{eqnarray*}
  \tilde E_i^Y(z,s,\e)&=&\begin{cases}E_i(z,s,\e)-\delta_{ij}(\Im{\sigma_j^{-1}z})^s-\Phi_{ij}(s,\e)(\Im{\sigma_j^{-1}z})^{1-s}\\ E_i(z,s,\e)\end{cases}\\
  \hat\f^Y(z,\e) &=&\begin{cases} \hat\f(z,\e)-a_j(\e)(\Im{\sigma_j^{-1}z})^{1-s} \\ \hat\f(z,\e)\end{cases}
\end{eqnarray*}
where for both functions the first case is $z\in \sigma_j \F_\infty^Y$ and the second is $z\in \F_0$.
 Taking the residue at $s=s_0(\e)$ and letting $Y\to\infty$ we find, using $\Re(s_0(\e))>1/2$ that
\begin{equation*}
  \inprod{ \hat\f(z,\e)}{\res_{s=s_0(\e)} E_j(z,s,\e)}=a_j(\e)
\end{equation*}
But by Assumption \ref{goaway} $\displaystyle\res_{s_0(\e)} E_\infty(z,s,\e)$ is zero for $\e$ sufficiently small, and hence $a_\infty(\e)=0$ when $\e$ is sufficiently small. This proves that $ S_\lambda L_\e P_j^{1}\f$ decays exponentially as $y\to\infty$.

We can now rigorize the inner product argument alluded to above:

We are using smoothly truncated Eisenstein series. Let $h:\R\to [0,1]$ be a smooth function which satisfies 
\begin{equation*}
  h(t)=\begin{cases}1&\textrm { if }t\leq 0\\0&\textrm { if }t\geq 1.\end{cases}
\end{equation*} Let $h_T(t):=h((t-T)/T)$
We then define the smoothly truncated Eisenstein series as the standard Eisenstein series minus a smooth cut-off of the zero Fourier coefficient, i.e. 
\begin{equation*}
  E_{sm}^T(z,s,\chi_0)=E_\infty(z,s,\chi_0)-h_T(y)(y^s+\Phi_{ij}(s,\chi)y^{1-s})
\end{equation*} if $\G z$ intersects $\F_\infty^Y$ (See \ref{funddomain}). This is well-defined for $T$ sufficiently large. The smoothly truncated Eisenstein series is in $D_{\G,\chi_0}$ as can be readily checked. 
Therefore, using the selfadjointness of $L(0)$ we see that 
\begin{equation}\label{gettingthere}
   \inprod{ (L(0)-\lambda)S_\lambda L_\e\f}{E_{sm}^T(z,s_0,\chi_0)}
\end{equation}
equals
\begin{equation*}
  \inprod{ S_\lambda L_\e\f}{(L(0)-\lambda)E_{sm}^T(z,s_0,\chi_0)}.
\end{equation*}
Since $E_{sm}^T(z,s_0,\chi_0)\in D_{\G,\chi_0}$ we have $L(0)E_{sm}^T(z,s_0,\chi_0)=\Delta E_{sm}^T(z,s_0,\chi_0)$. Therefore we may use (\ref{eisenstein-eigenfunction}) to conclude that $(L(0)-\lambda)E_{sm}^T(z,s_0,\chi_0)$ is non-zero only if $T\leq y \leq 2T$ , and in that case it equals
\begin{align*}
 -y^2h^{''}((y-T)/T)/T^2&(y^s+\Phi_{\infty\infty}(s)y^{1-s})\\ &+-2y^2h^{'}((y-T)/T)/T(sy^{s-1}+\Phi_{\infty\infty}(s)(1-s)y^{-s})
\end{align*}

which is $O(T^{s_0})$ when $T\to\infty$. (The implied constant depends on $\G$ and $s$)

Combining this estimate with the exponential decay of  $ S_\lambda L_\e P_j^{1}\f$ as $y\to\infty$ we easily find that  
\begin{equation}\label{goingtozero}
  \inprod{ S_\lambda L_\e\f}{(L(0)-\lambda)E_{sm}^T(z,s_0,\chi_0)}\to 0 
\end{equation}
as  $T\to\infty$.

To see that (\ref{goingtozero}) implies that (\ref{ceg}) is zero we note that the difference of (\ref{ceg}) and  (\ref{gettingthere}) equals
\begin{equation*}
  \int_{\F_\G} (L(0)-\lambda)S_\lambda L_\e\f(z)(\overline{E_\infty(z,s_0,\chi_0)-E_{sm}^T(z,s_0,\chi_0)}d\mu(z)
\end{equation*} which clearly goes to zero as $T\to\infty$ since  $(L(0)-\lambda)S_\lambda L_\e\f(z)=(1-P)L_\e\f$ decays exponentially at all cusps, and $E_\infty(z,s_0,\chi_0)-E_{sm}^T(z,s_0,\chi_0)$ is zero if $y<T$ and $O(y^{s_0})$ for $y\geq T$. It follows that 
\begin{align*}
\int_{\F_\G} L_\e&\f(z)(\overline{E_\infty(z,s_0,\chi_0)}d\mu(z)\\
  &= \int_{\F_\G} (L(0)-\lambda)  S_\lambda L_\e\f(z)(\overline{E_i(z,s_0,\chi_0)}d\mu(z)\\
&=\lim_{T\to\infty}\inprod{ S_\lambda L_\e\f}{(L(0)-\lambda)E_{sm}^T(z,s_0,\chi_0)}=0
\end{align*}
from which the result follows.
\end{proof}

\begin{remark} The main reason for Assumption \ref{goaway} is to allow to conclude (\ref{goingtozero}). We may draw the same conclusion if $S_\lambda L_\e\f_=O(y^{1-s_0-\delta})$ for some positive $\delta$. Hence any polynomial improvement of what may be concluded from Proposition \ref{faddeev-result} would make  our proof of the Selberg conjecture unconditional. Such an improvement is \emph{not} true in Proposition \ref{faddeev-result} for general groups nor on all of $B_0$. It is true on the set of cusp forms, but since $L_\e\f$ is not cuspidal that doesn't help much. 
\end{remark}

\section{Number theory}\label{number-theory}
We now specialize to a specific type  of arithmetic subgroups of $\slr$ namely the Hecke congruence groups with Dirichlet character. For these we review the theory of Hecke operators and automorphic twists by characters. Using twists by Dirichlet characters we will show that we can arrange for the Phillips-Sarnak integral (\ref{PS}) to be non-zero. This is the main arithmetic tool which will eventually lead to a proof of Theorem \ref{highlow}.

We emphasize that in this section we do not assume that $\lambda$ is a small eigenvalue, so we are not just  describing the subtleties of the empty set.
References for this section is  \cite{Atkin:1970aa, Miyake:2006aa, Bump:2003aa, Duke:2002aa, Iwaniec:2004aa, Stromberg:aa}. 

In  \cite{Duke:2002aa}  -- which gives the most comprehensive account for Maa\ss{} forms -- the character related to the group is assumed to be primitive. We cannot afford this luxury and we  make a review of the general situation. Our statements are close analogues of statements from \cite[Ch 4.]{Miyake:2006aa}. 

\subsection{Primitive forms and $L$-functions}\label{Dirichlet-setup}
Let $\G=\G_0(q)$ and let $\chi:\Z\to S^1$ be an \emph{even} Dirichlet character mod $q$. As usual this give rise to a group character on $\G_0(q)$ by setting 
\begin{equation}
  \chi'(\gamma)=\chi({d}),\qquad \g=\mattwo{a}{b}{c}{d} \in \G_0(q)
\end{equation}

We say that a Maa\ss{} cusp form of the automorphic Laplacian $A(\G_0(q),\chi')$ with eigenvalue $\lambda$  is of  of level $q$ and nebentypus $\chi$.
We denote the space of such forms by 
\begin{equation*}
  S_\lambda(q,\chi) .
\end{equation*}
Every Dirichlet character $\chi$ may be written uniquely as 
\begin{equation*}
  \chi=\chi^*\cdot\chi_0^q
\end{equation*}
where $\chi_0^q$ is the trivial character mod $q$, and $\chi^*$ is the primitive character mod $q^*$ inducing $\chi$, where $q^*\mid q$. We let $S_\lambda^\old(q,\chi)$ be the set of \emph{old}-forms i.e. the linear space generated functions  
\begin{equation*}
  \f(dz) \textrm{ where } dq'\mid q\,\quad q^*\mid q',\quad \f\in S_\lambda(q',\chi^*\cdot\chi_0^{q'})
\end{equation*}
and we let $S_\lambda^\new(q,\chi)$ be the orthogonal complement of $S_\lambda^\old(q,\chi)$ in $S(q,\chi)$, i.e.
\begin{equation*}
  S_\lambda^\new(q,\chi)=S_\lambda(q,\chi)\ominus S_\lambda^\old(q,\chi)
\end{equation*}

The Hecke operators $T_n:S_\lambda(q,\chi)\to S_\lambda(q,\chi)$ $n\in\N $ are defined by
\begin{equation*}
  T_nf(z)=\frac{1}{\sqrt{n}}\sum_{ad=n}\chi(a)\sum_{b \bmod d}f\left(\frac{az+b}{d}\right).
\end{equation*}
(See \cite{Atkin:1970aa}, \cite{Li:1975aa}, \cite{Miyake:2006aa}) and satisfy the relation
\begin{equation*}
  T_mT_n=\sum_{d\mid(m,n)}\chi(d) T_{mn/d^2}
\end{equation*}
In particular the Hecke operators commute. When $(q,n)=1$,  $T_n$ satisfy 
\begin{equation*}
  \inprod{T_nf}{g}=\inprod{f}{\overline{\chi(n)}T_ng},\qquad f,g\in S_\lambda(q,\chi)
\end{equation*}

The Hecke operators map $S^\new_\lambda(q,\chi)$ to itself, and if we restrict to newforms the Hecke operators have the \emph{multiplicity-one} property: Any two eigenfunctions of all $T_n$, $(n,q)=1$, with the same eigenvalues are equal up to multiplication by a scalar. Consequently  if $T:S^\new_\lambda(q,\chi)\to S^\new_\lambda(q,\chi)$ is a linear operator which commutes with all $T_n$, $(n,q)=1$ then every common eigenfunction for  $T_n$, $(n,q)=1$ is also an eigenfunction for $T$

We say that a non-zero newform $f\in S^\new_\lambda(q,\chi)$ is \emph{primitive} (of level $q$, nebentypus $\chi$, and eigenvalue $\lambda$) if it is an eigenfunction of all $T_n$, $(n,q)=1$. By the multiplicity one principle primitive forms are eigenfunction of $T_n$ for  \emph{all} $n\in\N$. The first Fourier coefficient $\rho_\f(1)$ is non-zero and we will therefore always assume that $\f$ is normalized such that $\rho_\f(1)=1$.  This normalization is called the \emph{Hecke normalization}.

\begin{remark}\label{inducedform}For $\f \in S_\lambda(q,\chi)$ not necessarily a newform, the following holds. If $\f$ is an eigenform for all $T_n$, $(n,q)=1$ then there exists a unique primitive  form $\f^{*}\in S^\new_\lambda(q',\chi')$ with $q'\mid q$, such that $\chi(n)=\chi'(n)$ and $\lambda_\f(n)=\lambda_{\f'}(n)$ when $(n,q)=1$. Here 
  \begin{equation*}
    T_n\f= \lambda_\f(n)\f.
  \end{equation*}

\end{remark}

We define the linear involution $V: S^\new_\lambda(q,\chi)\to S^\new_\lambda(q,\chi)$
\begin{equation*}
  V f(z)=f(-\overline z)
\end{equation*}
and the antilinear involution $W: S^\new_\lambda(q,\chi)\to S^\new_\lambda(q,\chi)$
\begin{equation*}
  Wf(z)=\overline{f\left((q\overline{z})^{-1}\right)}.
\end{equation*}
That $W$ maps $S_\lambda(q,\chi)$ to itself is not obvious (See \cite{Iwaniec:1997aa} p. 112). 
These operators satisfy 
\begin{align*}
  T_nV=&V T_n, &\textrm{ for }n\in \N \\
 T_nW=&\chi{(n)}WT_n, &\textrm{ if }(n,q)=1\\
\end{align*}
A primitive form $\f$ is automatically an eigenform of $V$ and $W$ with eigenvalue  $\varepsilon_\f$ and  $\eta_\f$  satisfying $\varepsilon_\f=\pm 1$,  $\abs{\eta_\f}=1$. A primitive form is called \emph{odd} if $\varepsilon_\f=-1$ and \emph{even} if $\varepsilon_\f=1$.

The Hecke operators act on the Fourier expansion at $i\infty$ (See (\ref{fourier-expansion-open}))

\begin{align*}
  T_n \f(z)&=T_n \sum_{m\neq 0}\rho_\f(n)\sqrt{y}K_{s-1/2}(2\pi \abs{m}y)e^{2\pi i mx}\\
&=\sum_{m\neq 0}\left(\sum_{d\mid (n,m)}\chi(d)\rho_\f(mn/d^2)\right)\sqrt{y}K_{s-1/2}(2\pi \abs{m}y)e^{2\pi i mx}
\end{align*}

Therefore the above mentioned properties forces the following relations among the Fourier coefficients which we state as a theorem for easy reference
\begin{theorem}\label{coefficients-properties}
Let  $\f$ be a Hecke-normalized primitive cuspidal Maa\ss{} form of level $q$ and nebentypus $\chi$ of conductor $m_\chi$. Then
\begin{enumerate}
\item \label{forste} $\lambda_\f(n) =\rho_\f(n)$ for $n\in\N$
\item  \label{anden} $\rho_\f(-n)=\varepsilon_\f\rho_\f(n)$ for $n\neq 0$ 
\item \label{tredje} $\lambda_\f(n)\lambda_\f(m)=\sum_{d\mid (m,n)}\chi(d)\lambda_\f(mn/d^2)$ 
\item \label{fjerde} If $p\mid q$ then $\lambda_\f(p^l)=\lambda_\f(p)^l$, and 
  \begin{equation*}
    \abs{\lambda_\f(p)}=\begin{cases}1&\textrm{if } p^k\parallel q,  p^k\parallel m_\chi\textrm{ for some }k \\
p^{-1/2}& \textrm{if } p\parallel q, p\nmid m_\chi\\
0&\textrm{ otherwise.}
\end{cases}
  \end{equation*}
\end{enumerate}
\end{theorem}
Here $p^k\parallel q$, means that $p^k$ devides $q$ but $p^{k+1}$ does not.

We remark that (\ref{fjerde}) is the analogue of \cite[Theorem 4.6.17]{Miyake:2006aa}. For a primitive form as above we put
\begin{equation}
  L(s,\f)=\sum_{n=1}^\infty\frac{\lambda_\f(n)}{n^s},\qquad \Re(s)>1.
\end{equation}
By the multiplicativity relations for the Hecke operators $L(s,\f)$ admits an Euler product expansion
\begin{equation}\label{euler}
  L(s,\f)=\prod_{p}(1-\lambda_\f(p)p^{-s}+\chi(p)p^{-2s})^{-1},\qquad \Re(s)>1.
\end{equation}
and is non-vanishing in this halfplane.
We define the completed $L$-function
\begin{align}
\begin{split}
\Lambda(s,\f)=\left(\frac{\sqrt{q}}{\pi}\right)^s\Gamma&\left(\frac{s+(s_\f-1/2)}{2}+\frac{1-\e_\f}{4}\right)\\ &\cdot\Gamma\left(\frac{s-(s_\f-1/2)}{2}+\frac{1-\e_\f}{4}\right) L(s,\f)
\end{split}
\end{align}
where $\lambda=s_\f(1-s_\f)$ is the Laplace eigenvalue of $\f$.
\begin{theorem}The completed $L$-function admits analytic continuation to an entire function , and it satisfies
  \begin{equation*}
    \Lambda(s,\f)=\omega_\f\overline{\Lambda(\overline{1-s},\f)}
  \end{equation*}
 where  $\omega_\f=\varepsilon_\f\overline{\eta_\f}$.
\end{theorem}

\subsection{Character twists of primitive forms}
In this section we recall the notion of character twists of Maa\ss{} forms, and state a theorem due to Rohrlich concerning such twists. We only need character twists of primitive forms, so we shall only consider such although the twists of general forms are only slightly more complicated.

Consider a primitive $\f\in S_\lambda^\new(q,\chi)$, and let $\psi$ be a \emph{primitive} Dirichlet character modulo $r$. 
 Then we define 
\begin{equation}
  \label{twists}
\f \times \psi=  \sum_{m\neq 0}\psi(m)\rho_\f(m)\sqrt{y}K_{s-1/2}(2\pi \abs{m}y)e^{2\pi i mx}
\end{equation}
By considering Gauss sums we find that 
\begin{equation}\f \times \psi=\tau(\overline{\psi})^{-1}\sum_{a \bmod r}\f(z+a/r)\in S_\lambda(N,\chi\psi^2)\end{equation}
where $\tau(\psi)$ denotes the Gauss sum of $\f$ and $N$ is the least common multiple of $q$, $q^*r$, and $r^2$ (remember that $q^*$ is the modulus of $\chi$). It is easy to see that $\f \times \psi$ is a non-zero eigenfunction of $T_n$ when $(n,N)=1$.
We denote by $\f \otimes \psi$ the unique newform giving rise to $\f \times \psi$ (See Remark \ref{inducedform}). In particular
\begin{equation}
  \label{eq:1}
  \rho_{\f \otimes \psi}(m)=\psi(m)\rho_\f(m),\qquad\textrm{when }(N,m)=1.
\end{equation}

We note that the parities of $\f$ and $\psi$ multiply i.e. 
\begin{equation}\label{parity-consideration}
  \varepsilon_{\f\otimes\psi} =\varepsilon_{\f}\cdot\psi(-1)
\end{equation}
so a twist by an odd character changes the parity of $\f$ while  a twist by an even character keeps the parity.

We state a simplified version of a theorem due to Rohrlich \cite{Rohrlich:1989aa}:
\begin{theorem}\label{rohrlich}
Let $\f$ be a primitive Maa\ss{} cusp form,  $s$  any  complex number, and  $M$  any integer. Then there exist infinitely many even primitive Dirichlet characters $\psi$ of conductor $m_\psi$ such that $(M,m_\psi)=1$ and
\begin{equation*}
  L(s,\f \otimes \psi)\neq 0.
\end{equation*}
\end{theorem}

\subsection{Modular forms of weight 2}\label{formsconstructed}
In this section we construct an element in $M_2^\infty(\G_0(q))$ which will be used to do character perturbation as in Section \ref{erc-application-is-over}.

Consider the holomorphic Eisenstein series of weight 2 for the 
modular group $\sl$, which may be defined as 
\begin{equation*}
E_2(z)=1-24\sum_{n=1}^\infty \sigma_1(n) e^{2\pi i n z},
\end{equation*}
where 
\begin{equation*}
\sigma_s(n)=\sum_{d\mid n}d^s.
\end{equation*}
The holomorphic Eisenstein series of weight 2 is only \lq quasi\rq-modular. More precisely it satisfies
\begin{equation*}
E_2(\g z)= (cz+d)^2E_2(z)-\frac{6i}{\pi}c(cz+d),\qquad\textrm{ for all }\g\in \sl.
\end{equation*}
But for any $q\in\N$ the following difference is modular and non-zero:
\begin{equation*}\label{eisensteinmodform}
  G_{q}(z):=E_2(z)-qE_2(qz)\in M_2(\G_0(q)).
\end{equation*}
The zero Fourier coefficent of  $G_{2,q}$ equals $1-q$. To construct a non-zero element in $M^\infty(\G_0(q))$ we assume that $q>1$ is not a prime and let $q=q_1q_1$ be a nontrivial factorization. We then define
\begin{equation*}
  G_{q_1,q_2}(z):=G_{q_1}(z)-G_{q_1}(q_2z)\in M_2^{\infty}(\G_0(q))
\end{equation*}
We note that this is nothing but a particular element in $\mathcal{E}(\G_0(q))$ the space of Eisenstein series of weight 2 for $\G_0(q)$ (See \cite[p. 21]{Sarnak:1990fk}, \cite{Scholl:1986aa}). The first Fourier coefficient of  $G_{q_1,q_2}(z)$ equals -24.


\subsection{The Phillips-Sarnak integral}
We already considered the integral
\begin{equation}\label{general-integral}
 I(s) =\int_{\F_\G}L_\e\f(z)\overline{E_\infty(z,\overline{s},\chi)}d\mu(z).
\end{equation}
from one point of view (Theorem \ref{integraliszero}). We now consider it from a number theory point of view. 
We start by taking a general eigenfunction $\f$ (i.e. not necessarily primitive), which is an eigenfunction of $V$, i.e. it is either even or odd. We assume that $f\in M_2(\Gamma_0(q))$ inducing $\chi_\e$ has Fourier coefficients at infinity
\begin{equation*}
 f(z)=\sum_{n=0}^\infty b_ne^{2\pi i n z} 
\end{equation*}
We note that $I(s)$ is well-defined even if $f$ is not cuspidal at infinity. Note however that without cuspidality at infinity we do not know if the conclusion of Theorem \ref{integraliszero} is valid.
\begin{theorem}\label{etskridt}Assume that $\f$ is either even or odd. If $\f$ is even then $I(s)=0$. If $\f$ is odd then 
  \begin{equation}
    \label{eq:6}
    I(s)=\frac{2}{2^{2s}\pi^{s-1}}\frac{\G(s+s_j)\G(s-s_j+1)}{\G(s)}
L(s+1/2,f\times \f) 
  \end{equation}
  where 
  \begin{equation*}
    L(s,f\times \f)=\sum_{n=1}^\infty\frac{b_n\rho_{\f}(n)}{n^{s}}
  \end{equation*}
is the Rankin-Selberg $L$-function of $f$ and $\f$.
\end{theorem}
\begin{proof}
  
For  $\Re(s)>1$ we may unfold the integral
\begin{align*}
   I(s)=&\int_{\F_\infty}L_\e\f(z)y^sd\mu(z)\\
=&4\pi i\int_{\F_\infty}y^2\left(f(z)\frac{\partial \f(z)}{\partial \overline z}+\overline{f(z)}\frac{\partial \f(z)}{\partial z}\right)y^sd\mu(z), \qquad \textrm{by (\ref{leavingsoon})}\\
=&2\pi i \int_{\F_\infty}y^2\left(f(z) \f(z)-\overline{f(z)}\f(z)\right)(-siy^{s-1})d\mu(z)\\
\intertext{We now use the Fourier expansions of $f$ and $\f$ and find}
=&\frac{s}{(2\pi)^{s-1/2} }\int_0^\infty e^{-y}K_{s-1/2}(y)y^{s-1/2}dy\sum_{n=1}^\infty\frac{b_n(\rho_{\f}(n)-\rho_\f(-n))}{n^{s+1/2}}    
\intertext{The integral can be evaluated (See \cite[6.621 3]{Gradshteyn:2000aa}) giving}
=&\frac{ s}{(2\pi)^{s-1/2} }\frac{\sqrt{\pi}}{2^{s+1/2}}\frac{\G(s+s_j)(\G(s-s_j+1))}{\G(s+1)}
\sum_{n=1}^\infty\frac{b_n(\rho_{\f}(n)-\rho_\f(-n))}{n^{s+1/2}}. 
\end{align*} 
We may now use Theorem \ref{coefficients-properties}  \ref{anden}) (which holds even though $\f$ is not primitive) to conclude that this is identically zero if $\varepsilon_\f=1$, i.e. if $\f$ is even.  This can also be seen directly from (\ref{general-integral}). In the odd case the claimed expression follows easily. 
\end{proof}
We will now show that if we further assume that  $\f$ is a primitive Hecke-normalized form of weight 0, eigenvalue $\lambda$, level $q$ and nebentypus $\chi$ and that the $f\in M_2(\G_0(q))$ giving rise to $L_\e$ (See \ref{perturbed}) equals $f(z)=G_{q_1,q_2}$ as defined in Section \ref{eisensteinmodform} then we can deduce the following theorem:
\begin{theorem} \label{integral-calculation}  On the above assumptions: If $\f$ is odd then 
  \begin{equation*}
  I(s)=-24\left(1-\frac{\lambda_\f(q_1)}{q_1^{s-1/2}}\right) \left(1-\frac{\lambda_\f(q_2)}{q_2^{s+1/2}}\right)\frac{\Lambda(s-1/2,\f)\Lambda(s+1/2,\f)}{\Lambda(2s,\chi)}.
\end{equation*}
Otherwise $I(s)=0$.
\end{theorem}
Here
\begin{equation}
  \Lambda(s,\chi)=\left(\frac{q}{\pi}\right)^{s/2}\G(s/2)\sum_{n=1}^\infty\frac{\chi(n)}{n^s}, \qquad \Re(s)>1
\end{equation} is the completed $L$-function of the even Dirichlet character $\chi$. We note that $\Lambda(s,\chi)$ only satisfies a standard functional equation if $\chi$ is primitive which we do not assume.

\begin{proof}
 We note that $f\in M_2(\G_0(q))$ giving $L_\e$  is of the form
\begin{equation}\label{form-assumption}
  f(z)=\sum_{d\mid q}t_dE_2(dz) ,\qquad t_d\in \R,
\end{equation}
 We may assume that $\f$ is odd. By Theorem \ref{etskridt} we need to consider the Rankin-Selberg $L$-function 
\begin{equation}\label{fundamentalsum}
  \sum_{n=1}^\infty\frac{b_n\lambda_{\f}(n)}{n^s} 
\end{equation}
which by the assumption (\ref{form-assumption}) equals
\begin{equation*}
- \sum_{d\mid q}t_d\sum_{n=1}^\infty \frac{24\sigma_1(n)\lambda_{\f}(dn)}{(dn)^s} 
\end{equation*}
From Theorem \ref{coefficients-properties} we conclude that $\lambda_\f(dn)=\lambda_\f(d)\lambda_\f(n)$ whenever $d\mid q$. Therefore (\ref{fundamentalsum}) equals 
\begin{equation*}
 -24 \sum_{d\mid q}\frac{t_d\lambda_\f(d)}{d^s} \sum_{n=1}^\infty \frac{\sigma_1(n)\lambda_{\f}(n)}{n^s}
\end{equation*}
The first sum equals
\begin{equation*}
  (1-\lambda_\f(q_1)q_1^{-s+1}) (1-\lambda_\f(q_2)q_2^{-s})
\end{equation*}
where we have used $\lambda_\f(q_1q_2)=\lambda_\f(q_1)\lambda_\f(q_2)$.
The last sum 
\begin{equation*}
  \sum_{n=1}^\infty \frac{\sigma_1(n)\lambda_{\f}(n)}{n^s}
\end{equation*}
may be calculated as follows: The sum factors into an Euler product
\begin{equation*}
  \prod_{p}\sum_{n=0}^\infty\frac{\sigma_1(p^n)\lambda_\f(p^n)}{p^{ns}}
\end{equation*}

It is straightforward to check - using Theorem \ref{coefficients-properties} (\ref{tredje}) --that each local factor equals
\begin{equation*}
  \frac{(1-\chi(p)p^{-(2s-1)})}{(1-\lambda_\f(p)p^{-(s-1)}+\chi(p)p^{-2(s-1)})(1-\lambda_\f(p)p^{-s}+\chi(p)p^{-2s})}
\end{equation*}
Therefore (compare \ref{euler}) we find that 
\begin{equation*}
  \sum_{n=1}^\infty \frac{\sigma_1(n)\lambda_{\f}(n)}{n^s}=\frac{L(s-1,\f)L(s,\f)}{L(2s-1,\chi)}
\end{equation*}
The result follows by comparing $\Gamma$ factors using the Legendre duplication formula.
\end{proof}

\section{Proof of Theorem \ref{highlow}}\label{cannot}
In this section we prove Theorem \ref{selbergconjecture}. 
We start by noticing that  Selberg's Conjecture \ref{selbergconjecture} implies Conjecture \ref{myconjecture}. This follows from the continuity of eigenvalues under character perturbation, or more precisely Corollary \ref{partytime}. For every $\lambda$, $a\leq \lambda\leq b$ Corollary \ref{partytime} gives that for small enough $\e$ there are no eigenvalues for the perturbed system in a neighborhood of $\lambda$. In particular there are no residual eigenvalues. Using compactness of $[a,b]$ we find that there is an $\e_0>0$ such that when $\abs{\e}\leq \e_0$ there are no residual eigenvalue for $A(\G,\e)$ in the whole interval $[a,b]$.

  Proving that Conjecture \ref{myconjecture} implies Conjecture \ref{selbergconjecture} is more involved. We notice that it is enough to prove non-existence of primitive Hecke-Maa\ss{} forms $\psi\in S^\new_\lambda(q,\chi)$ where $\lambda<1/4$, and $\chi$ is a Dirichlet character mod $N$. To see this we let $\G$ be any congruence group i.e. $\G(N)\subseteq \G\subseteq \G(1)$ for some $N$. If $\f\in S_\lambda(\G)$ then $\f(Nz)\in  S_\lambda(\G_1(N^2))$ (See \cite[p. 114]{Miyake:2006aa}). Then we use that for every positive integer $M$ 
$$ S_\lambda(\G_1(M))=\bigoplus_{\chi}  S_\lambda(\G_0(M),\chi)$$
where the sum is over all Dirichlet characters modulo $M$ (See  \cite[Lemma 4.3.1]{Miyake:2006aa}). Hence an eigenfunction for a congruence group induces at least one non-trivial primitive Hecke-Maa\ss{} form of some level and nebentypus with the same eigenvalue. We now prove that no primitive Maa\ss{} cusp forms exist with $0<\lambda<1/4$.

Assume that $\f$ is a non-trivial primitive Hecke-Maa\ss{} form of level $q$ and nebentypus $\chi$ with eigenvalue $s_0(1-s_0)<1/4$. By possibly twisting with an odd primitive character we may assume that $\f$ is odd (See \ref{parity-consideration}) and of level $q>1$ not a prime (See e.g. \cite[Proposition 14.20]{Iwaniec:2004aa}). By Theorem \ref{rohrlich} there exist an \emph{even} primitive Dirichlet character $\psi$ mod $r$ such that 
\begin{equation}\label{hereIcome}
  \Lambda(s_0-1/2,\f\otimes \psi )\neq 0.
\end{equation}
We note that by  (\ref{parity-consideration}) $\f\otimes \psi$ is still odd, and we denote its nebentypus by $\chi'$ and its level by $q'$. 

Since $q'>1$ is not a prime either we may choose a nontrivial factorization $q'=q_1q_2$, and we may form $f=G_{q_1,q_2}\in M_2^\infty(\G_0(q'))$  It follows  from Theorem \ref{coefficients-properties} \ref{fjerde}) that
\begin{equation*}
  \left(1-\frac{\lambda_{\f\otimes \psi}(q_1)}{q_1^{s-1/2}}\right) \left(1-\frac{\lambda_{\f\otimes\psi}(q_2)}{q_2^{s+1/2}}\right)\neq 0.
\end{equation*}
We have also 
\begin{equation*}
  \frac{\Lambda(s_0+1/2,\f\otimes \psi )}{\Lambda(2s_0,\chi')}\neq 0
\end{equation*}
since both $L$-functions are evaluated in the domain of absolute convergence. By Theorem \ref{integral-calculation}   we conclude that 
\begin{equation}\label{nonvanish}
  I(s_0)\neq 0.
\end{equation}
On the other hand on Conjecture \ref{myconjecture}  Theorem \ref{integraliszero}  tells us that  
\begin{equation*}
  I(s_0)= 0,
\end{equation*}
which contradicts (\ref{nonvanish}). Therefore the form $\f$ cannot
exist proving Conjecture \ref{selbergconjecture}. This concludes the
proof of Theorem \ref{highlow}.

\section{Further remarks}\label{further}
In this concluding section we make a few relatively straightforward remarks concerning the proof of Theorem \ref{highlow} which gives two more equivalent forms of Selberg's Conjecture \ref{selbergconjecture}. 
\subsection{Residual eigenvalues from infinity}
In our proof Theorem \ref{highlow}: To prove that Selberg's conjecture \ref{selbergconjecture} is implied by Conjecture \ref{myconjecture} we are only using non-existence of poles of Eisenstein series \emph{at infinity} and only character perturbations coming from $G_{q_1,q_2}$. Hence Selberg's Conjecture \ref{selbergconjecture} is also equivalent to the following: 
\begin{conjecture}\label{dumbass}
   For every $q>1$ non-prime and every Dirichlet character $\chi$ mod q there exist a non-trivial factorization $q=q_1q_2$ such that the following holds:  For every $1/2<c\leq d<1$ there exist $\e_0>0$ such that when $\abs{\e}<\e_0$ the Eisenstein series  $E_\infty(z,s,\e)$ is regular for $s\in [c,d]$. Here the character pertubation is given by (\ref{eq:2}) with $f=G_{q_1,q_2}$.
\end{conjecture}

Since poles of Eisenstein series occur only at poles of the diagonal of the scattering matrix we may formulate this in terms of the scattering term
\begin{equation*}
  \Phi_{\infty\infty}(s,\chi_\e).
\end{equation*}
 Hence Conjecture \ref{dumbass} states that for every $1/2<c\leq c<1$ there exist $\e_0>0$ such that when $\abs{\e}<\e_0$ 
$\phi_{\infty\infty}(s,\chi_\e)$ is regular for $s\in [c,d]$.

\subsection{Goldfeld Eisenstein series}
Consider $f \in M_2^\infty(\G_0(q))$ and $\chi$ a Dirichlet character mod $q$ as in Section \ref{Dirichlet-setup} . Goldfeld  \cite{Goldfeld:1999aa} introduced \emph{Eisenstein series series twisted with modular symbols}. We define the twisted Eisenstein series (for the cusp at infinity) as
\begin{equation}
  \label{eq:4}
  D^1(z,s,f,\chi)=\sum_{\g\in\GinfmodG}\overline{\chi'(\g)}\Re\left(\int_{i\infty}^{\g z}f(\tau)d\tau\right)\Im(\g z)^s
\end{equation}
for $\Re(s)>1$ (This definition is slightly different from Goldfeld's original one, but it is more convenient for our purpose). This series has meromorphic continuation to $s\in \C$. We refer to \cite{Goldfeld:1999aa, OSullivan:2000aa, Petridis:2004aa} for the basic properties of this series: All singularities are located at the spectral values of $A(\G,\chi')$. Lemma 2.14 of \cite{Petridis:2004aa} gives that $s=1$ is a removable singularity ( \cite[Lemma 2.14]{Petridis:2004aa} is correct although it is based on  \cite[Lemma 2.12]{Petridis:2004aa} which needs straightforward modifications). At cuspidal values it has at most a simple pole and the residue equals a linear combination of Phillips-Sarnak integrals (See \cite{Goldfeld:1999aa}, \cite[Theorem 1.1]{Petridis:2002ab}): 
\begin{equation}\label{polesofgoldfeldseries}
  \sum_{j=1}^m\int_{\F_\G}L_\e\phi_j(\tau)E_\infty(\tau,s_0,\chi)d\mu(\tau)\phi_j(z)
\end{equation}
where the sum is over a basis of the $s_0(1-s_0)$ eigenspace. On the basis of our proof of Theorem \ref{highlow} this naturally leads to the following conjecture:
\begin{conjecture}\label{minanden}

 For every $q>1$ non-prime and every Dirichlet character $\chi$ mod q there exist a non-trivial factorization  $q=q_1q_2$ such that the following holds: The Eisenstein series twisted with modular symbol $D^1(z,s,G_{q_1,q_2},\chi)$ is analytic in $\Re(s)>1/2$.
\end{conjecture}

This is certainly implied by Selberg's Conjecture \ref{selbergconjecture} and using (\ref{polesofgoldfeldseries}) we see that it implies that the Phillips-Sarnak integrals related to $s_0(1-s_0)$ are zero if $\Re(s_0)>1/2$. We may therefore use the same arguments as in Section \ref{cannot} to prove that Conjecture \ref{minanden} is in fact \emph{equivalent} to Selberg's Conjecture \ref{selbergconjecture}.

\nocite{Bruggeman:1994aa}
\nocite{Bruggeman:1986aa}

\bibliographystyle{plain}
\bibliography{minbibliography}
\end{document}